\documentclass[12pt]{amsart}
\usepackage{latexsym, amssymb, amsmath}

\newtheorem{thm}{Theorem}[section]
\newtheorem{lem}[thm]{Lemma}
\newtheorem{cor}[thm]{Corollary}
\newtheorem{prop}[thm]{Proposition}
\newtheorem{rem}[thm]{Remark}
\theoremstyle{definition}
\newtheorem{defn}[thm]{Definition}

\newcommand{\Ric}{{\rm Ric}}

\begin{document}

\title{Classification of almost Yamabe solitons in Euclidean spaces}


\author{Tatsuya Seko}
\address{Department of Mathematics,
 Shimane University, Nishikawatsu 1060 Matsue, 690-8504, Japan.}
\curraddr{}
\email{seko92903991@gmail.com}

\author{Shun Maeta}
\address{Department of Mathematics,
 Shimane University, Nishikawatsu 1060 Matsue, 690-8504, Japan.}
\curraddr{}
\email{shun.maeta@gmail.com~{\em or}~maeta@riko.shimane-u.ac.jp}
\thanks{The second author is  supported partially by the Grant-in-Aid for Young Scientists(B), No.15K17542, Japan Society for the Promotion of Science, and  partially by JSPS Overseas Research Fellowships 2017-2019 No. 70.}


\date{}

\dedicatory{}

\commby{}

\begin{abstract}
In this paper, we completely classify almost Yamabe solitons on hypersurfaces in Euclidean spaces arisen from the position vector field. Some results of almost Yamabe solitons with a concurrent vector field and almost Yamabe solitons on submanifolds in Riemannian manifolds equipped with a concurrent vector field are also presented. Moreover, we classify complete Ricci solitons on minimal submanifolds in non-positively curved space forms. For almost Yamabe solitons, all of results in this paper can be applied to Yamabe solitons.
\end{abstract}

\maketitle


\bibliographystyle{amsplain}

\section{Introduction}\label{intro} 
~\\
$(M,g,v,\rho)$ is called a Yamabe soliton if it satisfies
\begin{equation}\label{YS}
(R-\rho)g=\frac{1}{2}\mathcal{L}_vg,
\end{equation}
where $\mathcal{L}_vg$ is the Lie-derivative, $R$ is the scalar curvature of $M$ and $\rho$ is a constant.
If $\rho > 0, ~\rho = 0, ~\rho < 0,$ then a Yamabe soliton $(M,g,v,\rho)$ is called a shrinking, a steady or an expanding Yamabe soliton, respectively.

A Yamabe soliton $(M,g,v,\rho)$ is called a gradient Yamabe soliton 
if $v$ is the gradient of some function $f$ on $M$. We denote a gradient Yamabe soliton by $(M,g,f,\rho).$

Yamabe solitons are special solutions to the Yamabe flow. Under some conditions, Yamabe solitons have been studied (cf. \cite{2}, \cite{3} and \cite{4}). 
In particular, it is shown that any compact Yamabe soliton has constant scalar curvature (cf. \cite{2} and \cite{3}).

E. Barbosa and E. Ribeiro introduced a generalization of Yamabe solitons in \cite{BB13} as follows.
\begin{defn}[\cite{BB13}]
A Riemannian manifold $(M, g, v, \rho)$ is an almost Yamabe soliton if there exist a complete vector field $v$ and a smooth soliton function $\rho$ on $M$ satisfying 
\begin{equation}\label{AYS}
(R-\rho)g=\frac{1}{2}\mathcal{L}_vg.
\end{equation}
An almost Yamabe soliton $(M,g,v,\rho)$ is called a gradient almost Yamabe soliton if $v$ is the gradient of some function $f$ on $M$. We denote a gradient almost Yamabe soliton by $(M,g,f,\rho).$
\end{defn}
\begin{rem}
From the definition, if $\rho$ is constant, almost Yamabe solitons are Yamabe solitons.
\end{rem}

A vector field $v$ on $M$ is called a concurrent vector field if it satisfies
\begin{equation}\label{conc.vec.}
{\nabla}_Xv=X,
\end{equation}
for any vector field $X$ on $M$, where $\nabla$ is the Levi-Civita connection on $M$. One of the most important example of Riemannian manifolds with a concurrent vector field is Euclidean spaces. Because, the position vector field on Euclidean spaces satisfies $(\ref{conc.vec.})$. 
Riemannian manifolds endowed with concurrent vector fields have been studied (cf. \cite{1}, \cite{5} and \cite{6}).

In this paper, we completely classify almost Yamabe solitons on hypersurfaces in Euclidean spaces arisen from the position vector field $v$. 
We denote the tangential and the normal components of $v$ by $v^T$ and $v^{\perp}$, respectively.
\begin{thm}[Theorem~$\ref{main.2}$]\label{main}
Any almost Yamabe soliton $(M,g,v^T,\rho)$ on a hypersurface in a Euclidean space $\mathbb{E}^{n+1}$ is contained in either a hyperplane or a sphere.
\end{thm}

The remaining sections are organized as follows. Section~$\ref{Pre}$ contains some necessary definitions and preliminary geometric results.
In section~$\ref{concurrent}$, we show that any Yamabe soliton $(M,g,v,\rho)$ with a concurrent vector field $v$ is a gradient expanding Yamabe soliton with $\rho=-1$ and  the scalar curvature is zero. 
In section~$\ref{SubAYS}$, we consider almost Yamabe solitons on submanifolds in Riemannian manifolds endowed with a concurrent vector field.
Section~$\ref{Classification}$ is devoted to the proof of Theorem~$\ref{main}$. Finally in Appendix, we completely classify complete gradient Ricci solitons on minimal submanifolds in non-positively curved space forms.
\section{Preliminaries}\label{Pre} 

Let $(N,\tilde{g})$ be an $m$-dimensional Riemannian manifold and $(M,g)$ be an $n$-dimensional submanifold in $(N,\tilde{g})$.
We denote Levi-Civita connections on $(M,g)$ and $(N,\tilde{g})$ by $\nabla$ and $\tilde{\nabla}$ , respectively.

For any vector fields $X,Y$ tangent to $M$ and $\eta$ normal to $M$, the formula of Gauss is given by
\begin{equation*}
{\tilde{\nabla}}_XY={\nabla}_XY+h(X,Y),
\end{equation*}
where ${\nabla}_XY$ and $h(X,Y)$ are the tangential and the normal components of ${\tilde{\nabla}}_XY$.
The formula of Weingarten is given by
\begin{equation*}
{\tilde{\nabla}}_X\eta=-A_{\eta}(X)+D_X\eta,
\end{equation*}
where $-A_{\eta}(X)$ and $D_X\eta$ are the tangential and the normal components of ${\tilde{\nabla}}_X\eta$.
$A_{\eta}(X)$ and $h(X,Y)$ are related by
\begin{equation*}
g(A_{\eta}(X),Y)=\tilde{g}(h(X,Y),\eta).
\end{equation*}
The mean curvature vector $H$ of $M$ in $N$ is given by
\begin{equation*}
\displaystyle H=\frac{1}{n}~\text{trace}~h.
\end{equation*}
For any vector fields $X,Y,Z,W$ tangent to $M$, the equation of Gauss is given by
\begin{equation*}
\begin{tabular}{ll}
$\tilde{g}(\tilde{Rm}(X,Y)Z,W)=$ & $g(Rm(X,Y)Z,W)$ \vspace{0.3pc}\\
~ & $+\tilde{g}(h(X,Z),h(Y,W))$ \vspace{0.3pc}\\
~ & $-\tilde{g}(h(X,W),h(Y,Z)),$
\end{tabular}
\end{equation*}
where $Rm$ and $\tilde{Rm}$ are Riemannian curvature tensors of $M$ and $N$, respectively.
The equation of Codazzi is given by
\begin{equation*}
(\tilde{Rm}(X,Y)Z)^{\perp}=({\bar{\nabla}}_Xh)(Y,Z)-({\bar{\nabla}}_Yh)(X,Z),
\end{equation*}
where $(\tilde{Rm}(X,Y)Z)^{\perp}$ is the normal component of $\tilde{Rm}(X,Y)Z$ and ${\bar{\nabla}}_Xh$ is defined by
\begin{equation*}
({\bar{\nabla}}_Xh)(Y,Z)=D_Xh(Y,Z)-h({\nabla}_XY,Z)-h(Y,{\nabla}_XZ).
\end{equation*}
If $N$ is a space of constant curvature, then the equation of Codazzi reduces to
\begin{equation*}
0=({\bar{\nabla}}_Xh)(Y,Z)-({\bar{\nabla}}_Yh)(X,Z).
\end{equation*}

\section{Almost Yamabe solitons with a concurrent vector field}\label{concurrent}

Firstly, we show a formula of almost Yamabe solitons which is useful for study of almost Yamabe solitons.
\begin{lem}
\begin{equation}\label{f1}
(n-1)\Delta (R-\rho)+\frac{1}{2}~g(\nabla R,\nabla f)+R(R-\rho)=0.
\end{equation}
\end{lem}

\begin{proof}
Since
\begin{equation*}
\Delta {\nabla}_if={\nabla}_i\Delta f+R_{ij}{\nabla}_jf,
\end{equation*}
\begin{equation*}
\Delta {\nabla}_if={\nabla}_k{\nabla}_k{\nabla}_if={\nabla}_k((R-\rho)g_{ki})={\nabla}_i(R-\rho),
\end{equation*}
and
\begin{equation*}
{\nabla}_i\Delta f={\nabla}_i(n(R-\rho))=n{\nabla}_i(R-\rho),
\end{equation*}
we have
\begin{equation}\label{f1.1}
(n-1){\nabla}_i(R-\rho)+R_{ij}{\nabla}_jf=0,
\end{equation}
where $R_{ij}$ is the Ricci curvature of $M$.
By applying ${\nabla}_l$ to the both side of $(\ref{f1.1})$, we obtain
\begin{equation}\label{f1.2}
(n-1){\nabla}_l{\nabla}_i(R-\rho)+{\nabla}_lR_{ij} \cdot {\nabla}_jf+R_{ij}{\nabla}_l{\nabla}_jf=0.
\end{equation}
Taking the trace, we obtain $(\ref{f1})$.
\end{proof}

\begin{prop}\label{AYSwithC}
If an almost Yamabe soliton $(M,g,v,\rho)$ has a concurrent vector field $v$, then $M$ is a gradient almost Yamabe soliton with 
$R=\rho+1$.
\end{prop}

\begin{proof}

Firstly, we show that an almost Yamabe soliton with a concurrent vector field is a gradient almost Yamabe soliton.
Set
\begin{equation*}
\displaystyle f=\frac{1}{2}~g(v,v).
\end{equation*}
Then we have
\begin{equation*}
g(\nabla f,X)=X(f)=g(v,{\nabla}_Xv)=g(v,X) ,
\end{equation*}
for any vector field $X$ on $M$.

Secondly, we show that $R=\rho+1.$
Since $v$ is a concurrent vector field, we have
\begin{equation}\label{c.1}
\mathcal{L}_vg=2g.
\end{equation}
Combining $(\ref{c.1})$ with $(\ref{AYS})$, we obtain 
\begin{equation}
R=\rho +1.
\end{equation}
\end{proof}

\begin{cor}
Any compact almost Yamabe soliton with a concurrent vector field is a gradient expanding Yamabe soliton with zero scalar curvature and $\rho=-1$.
\end{cor}

\begin{proof}
By Proposition $\ref{AYSwithC}$ and $(\ref{AYS})$, we have $\Delta f=n$.
By applying maximum principle, we get $f$ is constant. From this and $(\ref{f1})$, we have $R=0$ and $\rho=-1$.
\end{proof}

By applying Proposition $\ref{AYSwithC}$ to Yamabe solitons, we can get the following. 
\begin{cor}
Any Yamabe soliton with a concurrent vector field is a gradient expanding Yamabe soliton with zero scalar curvature and $\rho=-1$.
\end{cor}

\begin{proof}
Since $\rho$ is constant, by Proposition $\ref{AYSwithC}$ and $(\ref{f1})$, we have $R=0$ and $\rho=-1$.
\end{proof}

\section{Almost Yamabe solitons on submanifolds}\label{SubAYS}

In this section, we assume that $(N,\tilde{g})$ is a Riemannian manifold endowed with a concurrent vector field $v$ and $(M,g)$ is a submanifold in $(N,\tilde{g})$.
We denote the tangential and the normal components of $v$ by $v^T$ and $v^{\perp}$, respectively.

To classify almost Yamabe solitons on a submanifold, we show the following Lemma which will be used in the proof of Proposition~$\ref{MAYS}$ and Theorem~$\ref{main.2}$. 

\begin{lem}\label{NSAYS}
Any almost Yamabe soliton $(M,g,v^T,\rho)$ on a submanifold $M$ in $N$ satisfies
\begin{equation}\label{ENSAYS}
(R-\rho-1)g(X,Y)=g(A_{v^{\perp}}(X),Y),
\end{equation}
for any vector fields $X, Y$ on $M$.
\end{lem}

\begin{proof}
Since $v$ is a concurrent vector field and by using formulas of Gauss and Weingarten, we have
\begin{equation}\label{s.1}
\begin{tabular}{ll}
\hspace{1.5pc}$X$ & $={\tilde{\nabla}}_Xv={\tilde{\nabla}}_X(v^T+v^{\perp})$ \vspace{0.5pc}\\
~ & $=({\nabla}_Xv^T+h(X,v^T))+(-A_{v^{\perp}}(X)+D_Xv^{\perp})$,
\end{tabular}
\end{equation}
for any vector field $X$ on $M$.
By comparing the tangential and the normal components of $(\ref{s.1})$, we obtain
\begin{equation}\label{s.2}
{\nabla}_Xv^T=X+A_{v^{\perp}}(X),~h(X,v^T)=-D_Xv^{\perp}.
\end{equation}
From the definition of Lie-derivative and $(\ref{s.2})$, we have
\begin{equation}\label{s.3}
(\mathcal{L}_{v^T}g)(X,Y)=2g(X,Y)+2g(A_{v^{\perp}}(X),Y),
\end{equation}
for any vector fields $X, Y$ on $M$.
Combining $(\ref{s.3})$ with $(\ref{AYS})$, we obtain $(\ref{ENSAYS})$.
\end{proof}

\begin{prop}\label{AYS is AGYS}
Any almost Yamabe soliton $(M,g,v^T,\rho)$ on a submanifold $M$ in $N$  is a gradient almost Yamabe soliton. 
\end{prop}

\begin{proof}
Set
\begin{equation*}
\displaystyle f=\frac{1}{2}~\tilde{g}(v,v).
\end{equation*}
Then we have
\begin{equation*}
g(\nabla f,X)=X(f)=\tilde{g}(v,{\tilde{\nabla}}_Xv)=\tilde{g}(v,X)=g(v^T,X),
\end{equation*}
for any vector field $X$ on $M$.
\end{proof}

\begin{prop}\label{MAYS}
If an almost Yamabe soliton $(M,g,v^T,\rho)$ on a submanifold $M$ in $N$ is minimal, then $R=\rho+1$.
\end{prop}

\begin{proof}
From Proposition~$\ref{AYS is AGYS}$, we know that any almost Yamabe soliton on a submanifold is a gradient almost Yamabe soliton.
Let $\{e_1,  \cdots , e_n\}$ be an orthonormal frame on $M$. From Lemma~$\ref{NSAYS}$, we have
\begin{equation*}
(R-\rho-1)g_{ij}=g(A_{v^{\perp}}(e_i),e_j).
\end{equation*}
Since $M$ is minimal and taking the trace, we obtain
\begin{equation*}
n(R-\rho-1)=n\tilde g(H,v^{\perp})=0. 
\end{equation*}
Therefore we conclude that
\begin{equation*}
R=\rho+1.
\end{equation*}
\end{proof}

\begin{cor}
Any compact almost Yamabe soliton on a minimal submanifold in $N$ is a gradient expanding Yamabe soliton with zero scalar curvature and $\rho=-1$.
\end{cor}

\begin{proof}
By Proposition $\ref{MAYS}$ and $(\ref{AYS})$, we have $\Delta f=n$.
By applying maximum principle, we get $f$ is constant. From this and $(\ref{f1})$, we have $R=0$ and $\rho=-1$.
\end{proof}

\begin{cor}
Any  Yamabe soliton on a minimal submanifold in $N$ is a gradient expanding Yamabe soliton with zero scalar curvature and $\rho=-1$.
\end{cor}

\begin{proof}
Since $\rho$ is constant, by Proposition $\ref{MAYS}$ and $(\ref{f1})$, we have $R=0$ and $\rho=-1$.
\end{proof}

\section{Classification of almost Yamabe solitons in Euclidean spaces}\label{Classification}

In this section, we give the proof of Theorem~$\ref{main}$, namely, we completely classify almost Yamabe solitons on hypersurfaces in Euclidean spaces arisen from the position vector field.
Let $v$ be the position vector field on Euclidean spaces.

\begin{thm}\label{main.2}
Any almost Yamabe soliton $(M,g,v^T,\rho)$ on a hypersurface in a Euclidean space $\mathbb{E}^{n+1}$ is contained in either a hyperplane or a sphere.
\end{thm}

\begin{proof}
Let $\alpha$ be a mean curvature and $\lambda$ be a support function of $M$,  i.e. $H=\alpha N$ and $\lambda =\tilde{g}(N,v)$ with a unit normal vector field $N$.
From Lemma~$\ref{NSAYS}$, 
\begin{equation*}
(R-\rho-1)g_{ij}=\tilde{g}(h(e_i,e_j),v^{\perp})=\tilde{g}({\kappa}_i g_{ij} N,v)={\kappa}_i g_{ij} \lambda ,
\end{equation*}
where $A_N(e_i)={\kappa}_ie_i, ~ i=1,\cdots ,n$. 
So we have
\begin{equation}\label{e.1}
R-\rho-1=\lambda {\kappa}_i.
\end{equation}
Taking the summation, we obtain
\begin{equation}\label{e.2}
R-\rho-1=\lambda \alpha.
\end{equation} 
Comparing $(\ref{e.1})$ and $(\ref{e.2})$, we have
\begin{equation*}
{\kappa}_i = \alpha .
\end{equation*}
Therefore $M$ is a totally umbilical submanifold with $A_N(e_i)=\alpha e_i$ and $h$ satisfies $h(X,Y)=\alpha g(X,Y) N$.
Now we have
\begin{equation*}
0={\tilde{\nabla}}_X(\tilde{g}(N,N))=2\tilde{g}({\tilde{\nabla}}_XN,N)=2\tilde{g}(D_XN,N).
\end{equation*}
Therefore $D_XN=0$.
So we obtain
\begin{equation*}
\begin{tabular}{rl}
$({\bar{\nabla}}_Xh)(Y,Z)=$ & $D_Xh(Y,Z)-h({\nabla}_XY,Z)-h(Y,{\nabla}_XZ)$ \vspace{0.5pc} \\
~$=$ & $X(\alpha ) g(Y,Z) N,$
\end{tabular}
\end{equation*}
for any vector fields $X, Y, Z$ on $M$.
From the equation of Codazzi, we have
\begin{equation*}
X(\alpha)Y=Y(\alpha)X.
\end{equation*}
Taking $X$ and $Y$ linearly independent, we conclude that $\alpha$ is a constant.\\
\underline{Case 1: $\alpha =0.$} 
From ${\tilde{\nabla}}_XN=0,$ N, restricted to $M$, is a constant in $\mathbb{E}^{n+1}$
and we have
\begin{equation*}
{\tilde{\nabla}}_X(\tilde{g}(v,N))=\tilde{g}({\tilde{\nabla}}_Xv,N)+\tilde{g}(v,{\tilde{\nabla}}_XN)=\tilde{g}(X,N)=0.
\end{equation*}
This shows that $\tilde{g}(v,N)$ is constant when $v$ and $N$ is restricted to $M$. Therefore $M$ is contained in the hyperplane normal to $N.$\\
\underline{Case 2: $\alpha \neq 0. $}
 We have
\begin{equation*}
{\tilde{\nabla}}_X(v+{\alpha}^{-1}N)=X+{\alpha}^{-1} {\tilde{\nabla}}_XN=X+{\alpha}^{-1} (-A_N(X))=0.
\end{equation*}
This shows that the vector field $v+{\alpha}^{-1}N$, restricted to $M$, is a constant in $\mathbb{E}^{n+1}.$ Therefore $M$ is contained in the sphere.
\end{proof}


\section{Appendix}
In this appendix, we completely classify complete gradient Ricci solitons on minimal submanifolds in Euclidean spaces or hyperbolic spaces.
$(M,g,f,\rho)$ is called a gradient Ricci soliton if it satisfies (see for example \cite{CLN06}),
\begin{equation}\label{RS}
\Ric+\nabla\nabla f +\frac{\rho}{2}g=0.
\end{equation}

Some recent progress on the subject can be found in \cite{Cao10}.

\begin{thm}
(i) Any complete gradient Ricci soliton on a minimal submanifold $M$ in a Euclidean space is an affine subspace.\\
(ii) There is no complete gradient Ricci soliton on a minimal submanifold $M$ in a hyperbolic space.
 \end{thm}

\begin{proof}
Let $c$ be a sectional curvature of a Euclidean space and a hyperbolic spaces, namely, $c=0$ or $-1$. From the equation of Gauss, 
$$\Ric(X,Y)=c(n-1) g(X,Y)-\tilde g (h(X,e_i),h(e_i,Y))+\tilde g (h(X,Y),nH),$$
for any vector fields $X, Y$ on $M$. Since $M$ is a minimal, we have
$$\Ric(X,Y)=c(n-1) g(X,Y)-\tilde g (h(X,e_i),h(e_i,Y)).$$
 Therefore,
 \begin{equation}\label{ap.1}
 R=cn(n-1)-|h|^2.
 \end{equation}
In \cite{BLCHEN09}, B. L. Chen showed that any complete gradient Ricci soliton has non-negative scalar curvature $R\geq0$.

\underline{Case 1: $c=0$.} By $(\ref{ap.1})$, $0\leq R=-|h|^2$. Therefore, $M$ is a totally geodesic submanifold and it is an affine subspace in $\mathbb{E}^{n+1}$.

\underline{Case 2: $c=-1$.} By the same argument, $|h|^2\leq -n(n-1)$, which can not happen.
 \end{proof}

{\bf Acknowledgments.}~
The work was done while the second author was visiting the Department of Mathematics of Texas A $\&$ M University-Commerce as a Visiting Scholar and he is grateful to the department and the university for the hospitality he had received during the visit. 
\bibliographystyle{amsbook}

\end{document}